\documentclass[11pt]{amsart}
\usepackage{amsfonts,latexsym,amsthm,amssymb,graphicx}
\usepackage[all]{xy}
\DeclareFontFamily{OT1}{rsfs}{}
\DeclareFontShape{OT1}{rsfs}{n}{it}{<-> rsfs10}{}
\DeclareMathAlphabet{\mathscr}{OT1}{rsfs}{n}{it}

\CompileMatrices

\newtheorem{theorem}{Theorem}[section]
\newtheorem{lemma}[theorem]{Lemma}
\newtheorem{corol}[theorem]{Corollary}
\newtheorem{prop}[theorem]{Proposition}

{\theoremstyle{definition} \newtheorem{defin}[theorem]{Definition}}
{\theoremstyle{remark} \newtheorem{remark}[theorem]{Remark}
\newtheorem{example}[theorem]{Example}
\newtheorem{warning}[theorem]{Warning}}

\newcommand{\Abb}{{\mathbb{A}}}
\newcommand{\Cbb}{{\mathbb{C}}}
\newcommand{\Qbb}{{\mathbb{Q}}}
\newcommand{\Lbb}{{\mathbb{L}}}
\newcommand{\Pbb}{{\mathbb{P}}}
\newcommand{\Zbb}{{\mathbb{Z}}}
\newcommand{\Tbb}{{\mathbb{T}}}

\newcommand{\cF}{{\mathscr F}}
\newcommand{\cL}{{\mathscr L}}
\newcommand{\cM}{{\mathscr M}}
\newcommand{\cN}{{\mathscr N}}
\newcommand{\cO}{{\mathscr O}}

\newcommand{\Til}[1]{{\widetilde{#1}}}
\newcommand{\one}{1\hskip-3.5pt1}
\newcommand{\bbchi}{\Psi}
\newcommand{\cfu}{{c_{\text{F}}}}
\newcommand{\csm}{{c_{\text{SM}}}}
\newcommand{\cma}{{c_{\text{Ma}}}}
\newcommand{\cwm}{{c_{\text{wMa}}}}

\newcommand{\cchi}{{\psi}}

\newcommand{\Var}{{\rm Var}}

\DeclareMathOperator{\codim}{codim}

\title{
Verdier specialization via weak factorization
}
\author{Paolo Aluffi}
\address{
Mathematics Department, 
Florida State University,
Tallahassee FL 32306, U.S.A.
}
\email{aluffi@math.fsu.edu}

\begin{document}

\begin{abstract}
Let $X\subset V$ be a closed embedding, with $V\smallsetminus X$ 
nonsingular. We define a constructible function $\cchi_{X,V}$ on $X$, 
agreeing with Verdier's specialization of the constant function $\one_V$ 
when $X$ is the zero-locus of a function on $V$. Our definition is given 
in terms of an embedded resolution of $X$; 
the independence on the choice of resolution is obtained as a consequence 
of the weak factorization theorem of \cite{MR1896232}. The main property of 
$\cchi_{X,V}$ is a compatibility with the specialization of the Chern class 
of the complement
$V\smallsetminus X$. With the definition adopted here, this is an easy 
consequence of standard intersection theory. It recovers Verdier's 
result when $X$ is the zero-locus of a function on $V$.

Our definition has a straightforward counterpart $\bbchi_{X,V}$ in a motivic 
group. The function~$\cchi_{X,V}$ and the corresponding Chern class 
$\csm(\cchi_{X,V})$ and motivic aspect $\bbchi_{X,V}$ all have natural
`monodromy' decompositions, for for any $X\subset V$ as above.

The definition also yields an expression for Kai Behrend's constructible function 
when applied to (the singularity subscheme of) the zero-locus of a function 
on $V$.
\end{abstract}

\maketitle


\section{Introduction}\label{intro}

\subsection{}
Consider a family $\pi: V\to D$ over the open disk, satisfying a
suitable condition of local triviality over $D\smallsetminus
\{0\}$. In~\cite{MR629126}, J.-L.~Verdier defines a `specialization
morphism' for constructible functions, producing a function
$\sigma_*(\varphi)$ on the central fiber~$X$ of the family for every
constructible function $\varphi$ on $V$. The key property of this
specialization morphism is that it commutes with the construction of
Chern classes of constructible functions in the sense of MacPherson
(\cite{MR50:13587}); cf.~Theorem~5.1 in Verdier's note. The
specialization morphism for constructible functions is induced from a
morphism at the level of constructible sheaves $\cF$, by taking
alternating sums of ranks for the corresponding complex of nearby
cycles $R\Psi_\pi \cF$.

The main purpose of this note is to give a more direct description of
the specialization morphism (in the algebraic category, over
algebraically closed fields of characteristic~$0$), purely in terms of
constructible functions and of resolution of singularities, including
an elementary proof of the basic compatibility relation with Chern
classes.  We will assume that $V$ is nonsingular away from $X$, and
focus on the case of the specialization of constant functions; by
linearity and functoriality properties, this suffices in order to
determine $\sigma_*$ in the situation considered by Verdier. On the
other hand, the situation we consider is more general than the
specialization template recalled above: we define a constructible
function $\cchi_{X,V}$ for {\em every\/} proper closed subscheme $X$
of a variety $V$ (such that $V\smallsetminus X$ is nonsingular), which
agrees with Verdier's specialization of the constant function $\one_V$
when $X$ is the fiber of a morphism from $V$ to a nonsingular curve.

The definition of $\cchi_{X,V}$ (Definition~\ref{maindef}) is
straightforward, and can be summarized as follows.  Let $w: W\to V$ be
a proper birational morphism such that $W$ is nonsingular, and
$D=w^{-1}(X)$ is a divisor with normal crossings and nonsingular
components, and for which $w|_{W\smallsetminus D}$ is an
isomorphism. Then define $\cchi_{D,W}(p)$ to be $m$ if $p$ is on a
single component of $D$ of multiplicity~$m$, and $0$ otherwise; and
let $\cchi_{X,V}$ be the push-forward of $\cchi_{D,W}$ to $X$.

Readers who are familiar with Verdier's paper \cite{MR629126} should
recognize that this construction is implicit in \S5 of that paper, if
$X$ is the zero-locus of a function on $V$. Our contribution is
limited to the realization that the weak factorization theorem of
\cite{MR1896232} may be used to adopt this prescription as a {\em
definition,\/} that the properties of this function follow directly
from the standard apparatus of intersection theory, and that this
approach extends the theory beyond the specialization situation
considered by Verdier (at least in the algebraic case).  Denoting by
$\csm(-)$ the Chern-Schwartz-MacPherson class of a constructible
function, we prove the following:\smallskip

\noindent{\bf Theorem I.}
Let $i: X\hookrightarrow V$ be an effective Cartier divisor. Then
\[
\csm(\cchi_{X,V})=i^* \csm(\one_{V\smallsetminus X})\quad.
\]

An expression for
$i^*\csm(\one_{V\smallsetminus X})$ in terms of the basic ingredients
needed to define $\cchi_{X,V}$ as above may be given as soon as
$i:X\to V$ is a regular embedding (cf.~Remark~\ref{regemb}). In fact,
with suitable positions, Theorem~I holds for arbitrary closed embeddings
$X\subset V$ (Theorem~\ref{spethm}).

Theorem~I reproduces Verdier's result when $X$ is a fiber of a
morphism from $V$ to a nonsingular curve; in that case (but not in
general) $\one_{V\smallsetminus X}$ may be replaced with $\one_V$, as
in Verdier's note. The definition of $\cchi_{X,V}$ is clearly
compatible with smooth maps, and in particular the value of
$\cchi_{X,V}$ at a point $p$ may be computed after restricting to an
open neighborhood of $p$. Thus, Verdier's formula for the
specialization function in terms of the Euler characteristic of the
intersection of a nearby fiber with a ball  (\S4 in \cite{MR629126}) may 
be used to compute $\cchi_{X,V}$ if $X$ is a divisor in $V$,
over $\Cbb$.
\smallskip

From the definition it is clear that the function $\cchi_{X,V}$ is
birationally invariant in the following weak sense:
\smallskip

\noindent{\bf Theorem II.}  
Let $\pi: V' \to V$ be a proper birational morphism; let
$X'=\pi^{-1}(X)$, and assume that $\pi$ restricts to an isomorphism
$V'\smallsetminus X' \to V\smallsetminus X$. Then
\[
\pi_* (\cchi_{X',V'})=\cchi_{X,V}\quad.
\]

\noindent (Here, $\pi_*$ is the push-forward of constructible functions.)

In fact, the whole specialization {\em morphism\/} commutes with
arbitrary proper maps, at least in Verdier's specialization situation
(\cite{MR629126}, Corollary~3.6). It would be desirable to establish
this fact for the morphism induced by $\cchi_{X,V}$ for arbitrary $X$,
by the methods used in this paper.\smallskip

The definition summarized above yields a natural decomposition of the
constructible function $\cchi_{V,X}$ (and hence of its Chern class
$i^*\csm(\one_{V\smallsetminus X})$) according to the multiplicities
of some of the exceptional divisors, see Remarks~\ref{distdecfun}
and~\ref{distdeccla}.  In the specialization situation, this
decomposition matches the one induced on the Milnor fiber by
monodromy, as follows from the description of the latter in
\cite{MR0371889}. As Sch\"urmann pointed out to me, an
analogous description in the more general case considered here may be
found in \cite{VeysProeyen}, Theorem~3.2.

\subsection{}
In the basic specialization situation, in which $X$ is the zero-scheme
of a function $f$ on~$V$ and $V$ is nonsingular, let $Y$ be the {\em
singularity subscheme\/} of $X$ (i.e., the `critical scheme'
of~$f$). One can define a constructible function $\mu$ on $X$ by
\begin{equation*}
\tag{*}
\mu=(-1)^{\dim V}(\one_X-\cchi_{X,V})\quad.
\end{equation*}
In this case (and over $\Cbb$), the function $\cchi_{X,V}$ agrees
with Verdier's specialization function~$\chi$ (here we use notation as
in \cite{MR2002g:14005}, cf.~especially Proposition~5.1). The function
$\mu$ is $0$ outside of $Y$, so may be viewed as a constructible
function on~$Y$. In fact, it has been observed
(cf.~e.g.,~\cite{MR1819626, MR2600874}) that the function $\mu$ is a
specific linear combination of local Euler obstructions, and in particular it
is determined by the scheme $Y$ and can be generalized to arbitrary
schemes. Kai Behrend denotes this generalization $\nu_Y$ in
\cite{MR2600874}. The definition of $\cchi_{X,V}$ given in this paper
yields an alternative computation of $\mu$ when~$Y$~is
the critical scheme of a function, and Theorem~II describes the
behavior through modifications along $X$ of this function: if $\pi: V' \to V$ 
is as in Theorem~II, then
\begin{equation*}
\tag{**}
\nu_Y=\pi_*(\nu_{Y'}) +(-1)^{\dim V}(\one_Y-\pi_*(\one_{Y'}))
\end{equation*}
provided that $V$ and $V'$ are nonsingular, $X$ and hence
$X'=\pi^{-1}(X)$ are hypersurfaces, and $Y$, $Y'$ are their
singularity subschemes. As $\cchi_{X,V}$ is defined for arbitrary
$X\subset V$, there may be a generalization of (*) linking Behrend's
function and $\chi_{X,V}$ when $X$ is not necessarily a hypersurface;
it would be interesting to have statements analogous to (**), holding
for more general $X$.

In \S\ref{wmather} we comment on the relation between
$\csm(\cchi_{X,V})$ and the `weighted Chern-Mather class' of the
singularity subscheme $Y$ of a hypersurface $X$; the degree of this
class is a Donaldson-Thomas type invariant (\cite{MR2600874}, \S4.3).
We also provide an explicit formula for the function $\mu$ in terms of
a resolution of the hypersurface $X$. It would be interesting to
extend these results to the non-hypersurface case.
\smallskip

In a different vein, J.~Sch\"urmann has considered the iteration of
the specialization operator over a set of generators for a complete
intersection $X\subset V$ (\cite{Schu}, Definition~3.6). It would be a
natural project to compare Sch\"urmann's definition (which depends on
the order of the generators) with our definition of $\cchi_{X,V}$
(which is independent of the order, and may be extended to arbitrary
$X\subset V$). Sch\"urmann also points out that the {\em deformation
to the normal cone\/} may be used to reduce an arbitrary $X\subset V$
to a specialization situation; this strategy was introduced in
\cite{MR737938}, and is explained in detail in~\cite{schup}, \S1.
Again, it would be interesting to establish the precise relation
between the resulting specialization morphism and the function
$\cchi_{X,V}$ studied here.\smallskip

\subsection{}
We include in \S\ref{mot} a brief discussion of a `motivic' invariant
$\bbchi_{X,V}$, also defined for any closed embedding $X\subset V$
into a variety, still assumed for simplicity to be nonsingular outside
of~$X$. This invariant can be defined in the quotient of the
Grothendieck ring of varieties by the ideal generated by the class of
a torus $\Tbb=\Abb^1-\Abb^0$, or in a more refined relative ring over
$X$.  The definition is again extremely simple, when given in terms of
a resolution in which the inverse image of $X$ is a divisor with
normal crossings; the proof that the invariant is well-defined also
follows from the weak factorization theorem.  As its constructible
function counterpart, the class $\bbchi_{X,V}$ admits a natural
`monodromy' decomposition (although a Milnor fiber is not defined in
general in the situation we consider), see Remark~\ref{distdecmot}.
When $V$ is nonsingular and $X$ is the zero-locus of a function
on~$V$, $\bbchi_{X,V}$ is a poor man's version of the Denef-Loeser
motivic Milnor/nearby fiber (\cite{MR1905328}, 3.5); it is defined
in a much coarser ring, but it carries information concerning the
topological Euler characteristic and some other Hodge-type data. We
note that $\bbchi_{X,V}$ is {\em not\/} the image of the limit of the
{\em naive\/} motivic zeta function of Denef-Loeser, since it does
carry multiplicity information, while $Z^{\text{naive}}(T)$ discards
it (see for example \cite{MR1905328}, Corollary~3.3). The limit of
the (non-naive) Denef-Loeser motivic zeta function $Z(T)$ encodes the
multiplicity and much more as actual monodromy information, and in
this sense it lifts the information carried by our~$\bbchi_{X,V}$. It
would be interesting to define and study an analogous lift for more
general closed embeddings $X\subset V$, and possibly allowing $V$ to
be singular along~$X$.

The approach of \cite{math.AG/0404512} could be used to unify the
constructions of $\cchi_{X,V}$ and $\bbchi_{X,V}$ given in this paper,
and likely extend them to other environments, but we will not pursue
such generalizations here since our aim is to keep the discussion at
the simplest possible level.  Likewise, `celestial' incarnations of
the Milnor fiber (in the spirit of \cite{MR2183846}, \cite{MR2280127})
will be discussed elsewhere.

\subsection{}
In our view, the main advantages of the approach taken in this paper
are the simplicity afforded by the use of the weak factorization
theorem and the fact that the results have a straightforward
interpretation for any closed embedding $X\subset V$, whether arising
from a specialization situation or not.  These results hold with
identical proofs over any algebraically closed field of characteristic
zero. We note that the paper \cite{VeysProeyen} of van Proeyen and
Veys also deals with arbitrary closed embeddings with nonsingular
complements, as in this note.  A treatment of Verdier specialization
over arbitrary algebraically closed fields of characteristic zero,
also using only the standard apparatus of intersection theory, was
given by Kennedy in \cite{MR1060893} by relying on the Lagrangian
viewpoint introduced by C.~Sabbah \cite{MR804052}.  Fu
(\cite{MR1129364}) gives a description of Verdier's specialization in
terms of normal currents.

We were motivated to take a new look at Verdier's specialization
because of applications to string-theoretic identities
(cf.~\cite{MR2495684}, \S4). Also, Verdier specialization offers an
alternative approach to the main result of \cite{delconCSM}.  The main
reason to allow $V$ to have singularities along $X$ is that this
typically is the case for specializations arising from pencils of
hypersurfaces in a linear system, as in these
applications. See~\S\ref{pencils} for a few simpler examples
illustrating this point.

\subsection{}
I am indebted to M.~Marcolli for helpful conversations, and I thank 
J. Sch\"urmann and W.~Veys for comments on a previous version of this paper.


\section{The definition}\label{chidef}

Our schemes are separated, of finite type over an algebraically closed
field~$k$ of characteristic~$0$. The characteristic restriction is due
to the use of resolution of singularities and the main result of
\cite{MR1896232}, as well as the theory of Chern-Schwartz-MacPherson
classes. (Cf.~\cite{MR1063344} and~\cite{MR2282409} for discussions of
the theory in this generality.)

\subsection{Definition of $\cchi_{X,V}$}
\begin{defin}\label{maindef}
Let $V$ be a variety, and let $i: X\hookrightarrow V$ be a closed
embedding.  We assume that $V\smallsetminus X$ is nonsingular and not
empty. The constructible function $\cchi_{X,V}$ on $X$ is defined as 
follows.
\begin{itemize}
\item Let $w: W\to V$ be a proper birational morphism; let 
$D=w^{-1}(X)$, and $d=w|_D$:
\[
\xymatrix{
D \ar@{^(->}[r]^j \ar[d]_d & W \ar[d]^w \\
X \ar@{^(->}[r]^i & V
}
\]
We assume that $W$ is nonsingular, $D$ is a divisor with normal
crossings and nonsingular components $D_\ell$ in $W$, and $w$
restricts to an isomorphism $W\smallsetminus D \to V\smallsetminus
X$. (Such a $w$ exists, by resolution of singularities.)  Let $m_\ell$
be the multiplicity of $D_\ell$ in~$D$.
\item We define a constructible function $\cchi_{D,W}$ on $D$ by
letting $\cchi_{D,W}(p)=m_\ell$ for $p\in D_\ell$, $p\not\in D_k$
($k\ne \ell$), and $\cchi_{D,W}(p)=0$ for $p\in D_\ell\cap D_k$, any
$k\ne \ell$.
\item Then let $\cchi_{X,V}:=d_*(\cchi_{D,W})$.
\end{itemize}
\end{defin}

We remind the reader that the push-forward of constructible functions
is defined as follows. For any scheme $S$, denote by $\one_S$ the
function with value $1$ along $S$, and $0$ outside of $S$. If $S$ is a
subvariety of $D$, and $x\in X$, $d_*(\one_S)(x)$ equals
$\chi(d^{-1}(x)\cap S)$.  By linearity, this prescription defines
$d_*(\varphi)$ for every constructible function $\varphi$ on
$D$. Here, $\chi$ denotes the topological Euler characteristic for
$k=\Cbb$; see \cite{MR1060893} or \cite{MR2282409} for the extension
to algebraically closed fields of characteristic~$0$.

Of course we have to verify that the definition of $\cchi_{X,V}$ given
in Definition~\ref{maindef} does not depend on the choice of $w: W\to
V$.

\begin{lemma}\label{spepro}
With notation as above, the function $\cchi_{X,V}$ is independent of
the choice of $w: W\to V$.
\end{lemma}

\begin{proof}
The weak factorization theorem of 
\cite{MR1896232} reduces the verification to the following fact.
\begin{itemize}
\item Let $W$, $D$, $\cchi_{D,W}$ as above;
\item Let $\pi: \Til W\to W$ be the blow-up of $W$ along a center
$Z\subseteq D$ that meets $D$ with normal crossings, and let $\widehat
D=\pi^{-1}(D)$;
\item Then $\pi_*(\cchi_{\widehat D,\Til W})=\cchi_{D,W}$, where
$\cchi_{\widehat D,\Til W}$ and $\cchi_{D,W}$ are defined by the
prescription for divisors with normal crossings given in
Definition~\ref{maindef}.
\end{itemize}
Recall that $Z$ meets $D$ with normal crossings if at each point $z$
of $Z$ there is an analytic system of parameters $x_1,\dots,x_n$ for
$D$ at $z$ such that $Z$ is given by $x_1=\cdots=x_{r+1}=0$, and $D$
is given by a monomial in the $x_\ell$'s. The divisor $\widehat
D=\pi^{-1}(D)$ is then a divisor with normal crossings, cf.~Lemma~2.4
in \cite{MR2098642}. The divisor $\widehat D$ consists of the proper
transforms $\Til D_\ell$ of the components $D_\ell$, appearing with
the same multiplicity~$m_\ell$, and of the exceptional divisor $E$,
appearing with multiplicity $\sum_{D_\ell\supset Z} m_\ell$. It is
clear that $\pi_*(\cchi_{\widehat D,\Til W})$ agrees with
$\cchi_{D,W}$ away from $Z$; we have to verify that the functions
match at all $z\in Z$. The fiber of $E=\pi^{-1}(Z)$ over $z$ is a
projective space of dimension $r=\codim_ZW-1$; we have to analyze the
intersection of the rest of $\pi^{-1}(D)$ (that is, of the proper
transforms $\Til D_\ell$) with this projective space.

Now there are two kinds of points $z\in Z$: either $z$ is in exactly
one component $D_\ell$, or it is in the intersection of several
components. In the first case, $\Til D_\ell$ is the unique component
of $\pi^{-1}(D)$ other than $E$ meeting the fiber $F$ of $E$ over
$z$. By definition, $\cchi_{\widehat D,\Til W}=m_\ell$ on the
complement of $F\cap \Til D_\ell$ in $\Til D_\ell$, and
$\cchi_{\widehat D,\Til W}=0$ along $F\cap \Til D_\ell$. Thus,
$\pi_*(\cchi_{\widehat D,\Til W})(z)$ equals
\[
m_\ell\cdot \chi(F\smallsetminus (F\cap \Til D_\ell))+ 0\cdot
\chi(F\cap \Til D_\ell) =m_\ell =\cchi_{D,W}(z)\quad,
\]
since $F\cong \Pbb^r$ and $F\cap \Til D_\ell\cong \Pbb^{r-1}$. This is
as it should.  If $z$ is in the intersection of two or more components
$D_\ell$, then $\cchi_{D,W}(z)=0$, so we have to verify that
$\pi_*(\cchi_{\widehat D,\Til W})(z)=0$. Again there are two
possibilities: either one of the components containing $z$ does not
contain $Z$, and then the whole fiber $F\cong \Pbb^r$ is contained in
the proper transform of that component, as well as in $E$; or all the
components $D_\ell$ containing $z$ contain $Z$. In the first case, the
value of $\cchi_{\widehat D,\Til W}$ is zero on the whole fiber $F$,
because $F$ is in the intersection of two components of~$\widehat D$;
so the equality is clear in this case.

In the second case, let $D_1,\dots,D_e$ be the components of $D$
containing $Z$; no other component of $D$ contains $z$, by
assumption. The proper transforms $\Til D_\ell$ meet the fiber $F
\cong \Pbb^r$ along $e$ hyperplanes meeting with normal crossings,
$1<e\le r+1$. The value of $\cchi_{\widehat D,\Til W}$ along~$F$ is
then $m_1+\cdots+ m_e$ along the complement $U$ of these hyperplanes,
and $0$ along these hyperplanes. Thus,
\[
\pi_*(\cchi_{\widehat D,\Til W})(z)=(m_1+\cdots+m_e)\cdot
\chi(U)\quad.
\]
The proof will be complete if we show that $\chi(U)=0$; and this is
done in the elementary lemma that follows.
\end{proof}

\begin{lemma}\label{elemlem}
Let $H_1,\dots, H_e$ be hyperplanes in $\Pbb^r$ meeting with normal
crossings, with $1\le e\le r+1$, and let $U=\Pbb^r\smallsetminus
(H_1\cup\cdots\cup H_e)$.  Then $\chi(U)=1$ for $e=1$, and $\chi(U)=0$
for $e=2,\cdots,r+1$.
\end{lemma}

\begin{proof}
The hyperplanes may be assumed to be coordinate hyperplanes, and hence
$U$ may be described as the set of $(x_0:\cdots:x_r)$ such that the
first $e$ coordinates are nonzero. As the first coordinate is nonzero,
we may set it to be~$1$, and view the rest as affine coordinates.  It
is then clear that $U\cong \Abb^{r+1-e}\times \Tbb^{e-1}$, where $\Tbb
=\Abb\smallsetminus \{0\}$ is the $1$-dimensional torus. The statement
is then clear, since $\chi(\Tbb)=0$.
\end{proof}

\begin{remark}
By definition, the value of $\cchi_{X,V}$ at a point $p\in X$ is
\[
\cchi_{X,V}=\sum_{\ell} m_\ell\, \chi(D_\ell^\circ\cap
w^{-1}(p))\quad,
\]
where $D_\ell^\circ=D_\ell\smallsetminus \cup_{k\ne \ell} D_k$.  In
the complex hypersurface case, this equals the Euler characteristic
$\chi(F_\theta)$ of the Milnor fiber, by formula (2) in Theorem~1 of
\cite{MR0371889}. An analogous interpretation holds in general, as
may be established by using Theorem~3.2 in \cite{VeysProeyen}.
\end{remark}

\begin{remark}\label{distdecfun}
Let $\alpha: \Zbb \to \Zbb$ be any function. The argument proving
Lemma~\ref{spepro} shows that one may define a constructible function
$\cchi_{X,V}^\alpha$ on $X$ as the push-forward of the function
$\cchi_{D,W}^\alpha$ with value $\alpha(m_\ell)$ on $D_\ell
\smallsetminus \cup_{k\ne \ell} D_k$ and $0$ on intersections, for $D$
and $W$ as in Definition~\ref{maindef}.

For $\alpha\ne$ identity we do not have an interpretation for
$\cchi_{X,V}^\alpha$ (or for the corresponding Chern class,
cf.~Remark~\ref{distdeccla}).  Letting $\epsilon_m$ be the function
that is $1$ at $m\in \Zbb$ and $0$ at all other integers, we have a
decomposition of the identity as $\sum_m m\, \epsilon_m$, and hence a
distinguished decomposition
\[
\cchi_{X,V}=\sum_m m\, \cchi_{X,V}^{\epsilon_m}\quad.
\]
The individual terms in this decomposition are clearly preserved by
the morphisms considered in Propositions~\ref{birinv}
and~\ref{smooth}. In the hypersurface case they can be related to the
monodromy action on the Milnor fiber, as the nonzero contributions
arise from the nonzero eigenspaces of monodromy (cf.~\cite{MR0371889},
Theorem~4; and Theorem~3.2 in~\cite{VeysProeyen} for generalizations
to the non-hypersurface case). From this perspective, the piece
$\cchi_{X,V}^{\epsilon_1}$ may be thought of as the `unipotent part'
of $\cchi_{X,V}$.
\end{remark}

\subsection{Basic properties and examples}
Theorem~II from the introduction is an immediate consequence of the
definition of~$\cchi_{X,V}$:

\begin{prop}\label{birinv}
Let $V'$ be a variety, and let $\pi: V' \to V$ be a proper
morphism. Let $X'=\pi^{-1}(X)$, and assume that $\pi$ restricts to an
isomorphism $V'\smallsetminus X' \to V\smallsetminus X$. Then
\[
\pi_* (\cchi_{V',X'})=\cchi_{V,X}\quad.
\]
\end{prop}

\noindent Indeed, a resolution for the pair $X'\subset V'$ as in
Definition~\ref{maindef} is also a resolution for the pair $X\subset
V$.  Another immediate consequence of the definition is the behavior
with respect to smooth maps:

\begin{prop}\label{smooth}
Let $U$ be a variety, and let $\eta: U \to V$ be a smooth morphism. Then
\[
\cchi_{\eta^{-1}(X),U}=\eta^*(\cchi_{X,V})\quad.
\]
\end{prop}

\noindent (Pull-backs of constructible functions are defined by
composition.)  Indeed, in this case one can construct compatible
resolutions.

For example, the value of $\cchi_{X,V}$ at a point $p\in X$ may be
computed by restricting to an open neighborhood of $p$.

\begin{example}\label{sci}
If $X$ and $V$ are both nonsingular varieties, then
$\cchi_{X,V}=\codim_XV\cdot \one_X$.

Indeed, if $X$ and $V$ are nonsingular, then we can let $w:W\to V$ be
the blow-up of $V$ along~$X$; $D$ is the exceptional divisor, and the
push-forward of $\one_D$ equals $\codim_XV\cdot \one_X$ because the
fibers of $d$ are projective spaces $\Pbb^{\codim_XV-1}$.
\end{example}

\begin{example}\label{nonred}
The function $\cchi_{X,V}$ depends on the scheme structure on $X$.
For example, let $X\subset \Pbb^2$ be the scheme defined by
$(x^2,xy)$, consisting of a line $L$ with an embedded point at~$p\in
L$. Then $\cchi_{X,\Pbb^2}$ equals $1$ along $L\smallsetminus \{p\}$,
while $\cchi_{X,\Pbb^2}(p)=2$.  (Blow-up at $p$, then apply
Definition~\ref{maindef}. In terms of the decomposition in
Remark~\ref{distdecfun}, $\cchi_{X,\Pbb^2}$ is the sum of
$\cchi_{X,\Pbb^2}^{\epsilon_1}=\one_{L\smallsetminus p}$ and $2
\cchi_{X,\Pbb^2}^{\epsilon_2}=2\cdot\one_p$.)  Thus,
$\cchi_{X,\Pbb^2}=\one_L+\one_p$, while of course
$\cchi_{L,\Pbb^2}=\one_L$.
\end{example}

\begin{example}
Let $X$ be the reduced scheme supported on the union of three
non-coplanar concurrent lines in $\Pbb^3$; for example, we may take
$(xy,xz,yz)$ as a defining ideal. Then $\cchi_{X,\Pbb^3}=2$ along the
three components, and $\cchi_{X,\Pbb^3} (p)=0$ at the point of
intersection~$p$.

To see this, blow-up $\Pbb^3$ at $p$ first, and then along the proper
transforms of the three lines, to produce a morphism $w:W\to \Pbb^3$
as in Definition~\ref{maindef}; $D=w^{-1}(X)$ consists of $4$
components, one of which dominates $p$. This component is a $\Pbb^2$
blown up at three points; the complement in this component of the intersection
with the three other exceptional divisors has Euler characteristic~$0$, so the 
push-forward of $\cchi_{D,W}$ equals $0$ at $p$.

If the three concurrent lines are coplanar, say with ideal
$(xy(x+y),z)$, then the value of $\cchi_{X,\Pbb^3}$ at the
intersection point is $-2$.

That the value of the function is $2$ away from the point of
intersection is clear {\em a priori\/} in both cases by
Example~\ref{sci}, since as observed above we may restrict to an open
set avoiding the singularity and use Example~\ref{sci}.
\end{example}

\subsection{A motivic invariant}\label{mot}
Definition~\ref{maindef} has a counterpart in a quotient of the naive
Grothendieck group of varieties $K(\Var_k)$. Recall that this is the
group generated by isomorphism classes of $k$-varieties, modulo the
relations $[S]=[S-T]+[T]$ for every closed embedding $T\subseteq S$;
setting $[S_1]\cdot [S_2] =[S_1\times S_2]$ makes $K(\Var_k)$ into a
ring. Denote by $\Lbb$ the class of $\Abb^1$ in $K(\Var_k)$, and by
$\Tbb=\Lbb-1$ the class of the multiplicative group of $k$.

We let $\cM^\Tbb$ denote the quotient $K(\Var_k)/(\Tbb)$. Every pair
$X\subset V$ as in Definition~\ref{maindef} (that is, with $V$ a
variety and $V\smallsetminus X$ nonsingular) determines a well-defined
element $\bbchi_{X,V}$ of $\cM^\Tbb$, as follows:
\begin{itemize}
\item Let $D,d,W,w$ be as in Definition~\ref{maindef};
\item For every component $D_\ell$ of $D$, let $D_\ell^\circ$ be the
complement $D_\ell\smallsetminus \cup_{k\ne l} D_k$;
\item Then $\bbchi_{X,V}=\sum_\ell m_\ell [D_\ell^\circ]$, where
$m_\ell$ is the multiplicity of $D_\ell$ in $D$.
\end{itemize}

The argument given in Lemma~\ref{spepro} proves that $\bbchi_{X,V}$ is
well-defined in the quotient $\cM^\Tbb$.  Indeed, the argument applies
word-for-word to show that the class of the fiber of a blow-up over
$z\in Z$ agrees mod $\Tbb$ with the class of $z$; Lemma~\ref{elemlem}
is replaced by the analogous result in $\cM^\Tbb$:

\begin{lemma}\label{elemlem2}
Let $H_1,\dots, H_e$ be hyperplanes in $\Pbb^r$ meeting with normal
crossings, with $1\le e\le r+1$, and let $U=\Pbb^r\smallsetminus
(H_1\cup\cdots\cup H_e)$.  Then $[U]=1\in\cM^\Tbb$ for $e=1$, and
$[U]=0\in\cM^\Tbb$ for $e=2,\cdots,r+1$.
\end{lemma}

The proof of Lemma~\ref{elemlem} implies this statement, as it shows
that $U\cong \Abb^{r+1-e}\times \Tbb^{e-1}$. It is also evident that
$\bbchi_{X,V}$ satisfies the analogues of Propositions~\ref{birinv}
and~\ref{smooth}.  Since $\chi(\Tbb)=0$, the information carried by an
element of $\cM^\Tbb$ suffices to compute topological Euler
characteristics, but is considerably more refined: for example, the
series obtained by setting $v=u^{-1}$ in the Hodge-Deligne polynomial
can be recovered from the class in $\cM^\Tbb$.

\begin{example}
Let $X\subseteq \Pbb^3$ be a cone over a smooth plane curve
$C\subseteq \Pbb^2$ of degree~$m$.  Blowing up at the vertex $p$
yields a resolution as needed in Definition~\ref{maindef}; the
function $\cchi_{X,\Pbb^3}$ is immediately computed to be $1$ outside
of the vertex and $m\cdot (\chi(\Pbb^2)-\chi(C))=(m-1)^3+1$ at $p$;
the class $\bbchi_{X, \Pbb^3}$ equals $[C]+m\cdot(3-[C])\in
\cM^\Tbb$. For $m>2$, this is not the class of a constant.
\end{example}

If $X$ is the scheme of zeros of a nonzero function $f: V \to k$, and
$V$ is nonsingular, then a {\em motivic Milnor fiber\/} $\psi_f$ was
defined and studied by J.~Denef and F.~Loeser,
see~\cite{MR1923998},~\S3. The motivic Milnor fiber $\psi_f$ is
defined in a ring $\cM_{k, loc}^{mon}$ analogous to the ring
$K(\Var_k)$ considered above, but localized at $\Lbb$ and including
monodromy information.  The much naiver $\bbchi_{X,V}$ generalizes to
arbitrary pairs $X\subset V$ (with $V\smallsetminus X$ nonsingular) a
small part of the information carried by the Denef-Loeser motivic
Milnor fiber in the specialization case.

\begin{remark}
It is probably preferable to work in the {\em relative\/} Grothendieck
ring of varieties over~$X$, mod-ing out by classes of varieties $Z\to
X$ which fiber in tori.  Lemma~\ref{elemlem2} shows that the resulting
class $\bbchi_{X,V}^{rel}$ is also well-defined. For $p\in X$, the
fiber of $\bbchi_{X,V}^{rel}$ over $p$ is well-defined as a class in
$\cM^\Tbb$, so it has an Euler characteristic, which clearly equals
$\cchi_{X,V}(p)$. Thus, the constructible function $\cchi_{X,V}$ may
be recovered from the relative class $\bbchi_{X,V}^{rel}$.  This point
of view is developed fully for Chern classes and more in
\cite{MR2201957}, \cite{math.AG/0404512}.  (Cf.~\cite{MR1905328},
\cite{MR2106970} for the relative viewpoint on the Denef-Loeser
motivic Milnor fiber.)
\end{remark}

\begin{remark}\label{distdecmot}
There is a class $\bbchi_{X,V}^\alpha$ (resp.,
$\bbchi_{X,V}^{rel,\alpha}$) for every function $\alpha: \Zbb \to
\Zbb$, defined by $\sum_\ell \alpha(m_\ell) [D_\ell^\circ]$ on a
resolution. In particular, there is a decomposition
\[
\bbchi_{X,V}=\sum_m m\, \bbchi_{X,V}^{\epsilon_m}
\]
with $\epsilon_m$ as in Remark~\ref{distdecfun}. As observed in that
remark, in the hypersurface case this decomposition can be related
with the monodromy decomposition. If $X$ is the zero scheme of a
function on $V$, and $V$ is nonsingular, then the limit of the {\em
naive motivic zeta function\/} $Z^\text{naive}(T)$ of Denef-Loeser
(\cite{MR1905328}, Corollary~3.3) lifts $\bbchi_{X,V}^{1}$, where
$\alpha\equiv 1$ is the constant function~$1$. In fact, the
corresponding expression
\[
\sum_{|I|>0} (-\Tbb)^{|I|-1} [D^\circ_I]
\]
(where $D_I=\cap_{\ell \in I} D_\ell$, and $D^\circ_I=D_I
\smallsetminus \cup_{\ell\not\in I} D_\ell$) may be verified to be a
well-defined element of $K(\Var_k)$ for any $X\subset V$ such that
$V\smallsetminus X$ is nonsingular, and with $D$ as in
Definition~\ref{maindef}. This also follows from the weak
factorization theorem; no localization is necessary. Also, this
expression is clearly preserved by morphisms which restrict to the
identity on $V\smallsetminus X$ (as in Theorem~II). Thus,
$\bbchi_{X,V}^{1}$ admits a natural lift to $K(\Var_k)$ in general. We
do not know whether $\bbchi_{X,V}$ itself admits a natural lift to
$K(\Var_k)$.
\end{remark}


\section{Compatibility with Chern classes}\label{mainpro}

\subsection{Specialization in the Chow group and Chern-Schwartz-MacPherson 
classes}\label{basicca} 
Recall (\cite{MR35:3707}, \cite{MR32:1727}, \cite{MR50:13587}) that
every constructible function $\varphi$ determines a class in the Chow
group, satisfying good functoriality properties and the normalization
restriction of agreeing with the total Chern class of the tangent
bundle if applied to the constant function~$\one$ over a nonsingular
variety.  We call this function the {\em Chern-Schwartz-MacPherson\/}
(CSM) class of $\varphi$, $\csm(\varphi)$. The functoriality may be
expressed as follows: if $f: X \to Y$ is a proper morphism, and
$\varphi$ is a constructible function on $X$, then
$f_*(\csm(\varphi))=\csm(f_*(\varphi))$.  The simplest instance of
this property is the fact that for every complete variety $X$
(nonsingular or otherwise), the degree of $\csm(\one_X)$ equals the
Euler characteristic of $X$.

We are particularly interested in the CSM classes of the function
$\cchi_{X,V}$ defined in \S\ref{chidef}, and of the function
$\one_{V\smallsetminus X} =\one_V-\one_X$, with value $0$ on $X$ and
$1$ along the complement of $X$.

In the situation considered by Verdier (and more generally when $X$ is
e.g., a Cartier divisor in $V$), there is a natural way to specialize
classes defined on $V$ or $V\smallsetminus X$ to $X$.  We consider the
following definition of the specialization of a specific class, for
arbitrary $X\subset V$.

\begin{defin}\label{specdefspec}
Let $V$ be a variety, $X\subset V$ a closed embedding, and assume that
$V\smallsetminus X$ is nonsingular.  Let $\tilde v:\Til V \to V$ be
the blow-up of $V$ along $X$, and let $\iota: E\to \Til V$ be the
exceptional divisor, $e: E\to X$ the induced map. We define the
`specialization of $\csm(\one_{V\smallsetminus X})$ to $X$' to be the
class
\[
\sigma_{X,V}(\csm(\one_{V\smallsetminus X}))
:=e_* \iota^* (\csm(\one_{\Til V\smallsetminus E}))\in A_*X\quad.
\]
\end{defin}

The blow-up $\Til V\to V$ in Definition~\ref{specdefspec} may be
replaced with any proper birational morphism $v': V' \to V$ dominating
the blow-up and restricting to an isomorphism on $V'\smallsetminus
{v'}^{-1}(X)$; this follows easily from the projection formula and the
functoriality of CSM classes.

When $X$ is a Cartier divisor in $V$, this specialization is the
ordinary pull-back, and could be defined for arbitrary classes in $V$.

\begin{lemma}\label{spediv}
Let $X\hookrightarrow V$ be a Cartier divisor. Then $\sigma_{X,V} 
(\csm(\one_{V\smallsetminus X}))=i^*\csm(\one_{V\smallsetminus X})$.
\end{lemma}

\begin{proof}
In this case $\Til V=V$, $E=X$, and $e$ is the identity.
\end{proof}

If $X\subset V$ is {\em not\/} a Cartier divisor, it is not clear how
to extend Definition~\ref{specdefspec} to arbitrary classes; but this
is not needed for the results in this paper.

By Lemma~\ref{spediv}, Theorem~I from the introduction
follows from the following result.

\begin{theorem}\label{spethm}
With notation as above,
\[
\csm(\cchi_{X,V})=\sigma_{X,V}( \csm(\one_{V\smallsetminus X}))\quad.
\]
\end{theorem}

\begin{proof}
First consider the case in which $V=W$ is nonsingular, and $X=D$ is a
divisor with normal crossings and nonsingular components $D_\ell$,
appearing with multiplicity $m_\ell$. Let $D_\ell^\circ\subseteq
D_\ell$ be the complement of $\cup_{k\ne\ell} D_k$ in $D_\ell$. By
definition of $\cchi_{D,W}$, and by linearity of the $\csm$ operator,
\[
\csm(\cchi_{D,W})=\sum_\ell m_\ell\, \csm(\one_{D_\ell^\circ})\quad.
\]
The key to the computation is the fact that the CSM class of the
complement of a divisor with normal crossings in a nonsingular variety
is given by the Chern class of a corresponding logarithmic twist of
the tangent bundle; for this, see e.g.~(*) at the top of p.~4002 in
\cite{MR2001i:14009}.  As $D_\ell^\circ$ is the complement of
$\cup_{k\ne \ell} D_k\cap D_\ell$, a normal crossing divisor in
$D_\ell$, we have (omitting evident pull-backs)
\[
\csm(\one_{D_\ell^\circ})=\frac{c(TD_\ell)}{\prod_{k\ne\ell}
(1+D_k)}\cap [D_\ell] =\frac{c(TW)}{\prod_k (1+D_k)}\cap
[D_\ell]\quad.
\]
Therefore
\[
\csm(\cchi_{D,W})=\frac{c(TW)}{\prod_k (1+D_k)}\cap \sum_\ell
m_\ell\,[D_\ell] =\frac{c(TW)}{\prod_k (1+D_k)}\cap [D]\quad.
\]
Since $\csm(\one_{W\smallsetminus D})=\frac{c(TW)}{\prod_k
(1+D_k)}\cap [W]$, this shows that
\begin{equation*}
\tag{*}
\csm(\cchi_{D,W})=j^* \csm(\one_{W\smallsetminus D})\quad,
\end{equation*}
where $j: D\hookrightarrow W$ is the inclusion. This is the statement
in the normal crossing case.

Now assume $X\subset V$ is any closed embedding, and apply (*)
to $D$, $W$ as in Definition~\ref{maindef},
observing that any such $w:W\to V$ must factor through the blow-up
$\Til V$ along $X$. Let $w': W\to \Til V$, $d': D \to E$ be the
induced morphisms, so that $w=\tilde v\circ w'$, $d=e\circ d'$:
\[
\xymatrix@R=15pt@C=15pt{
& D \ar[rr]^j \ar[ld]_{d'} \ar@{-}[d] & & W \ar[ld]_(.6){w'} \ar[dd]^(.6)w \\
E \ar[rr]^(.7)\iota \ar[dr]_(.4)e & \ar[d]^(.4)d & \Til V \ar[dr]_(.3){\tilde v} \\
& X \ar[rr]^(.4)i & & V
}
\]
By the covariance of $\csm$, we have
\begin{align*}
\csm(\cchi_{E,\Til V}) &=\csm(d'_* \cchi_{D,W}) \\
&= d'_* \csm(\cchi_{D,W}) \\
&= d'_* j^* \csm(\one_{W\smallsetminus D}) \\
&\overset\dagger= \iota^* w'_* \csm(\one_{W\smallsetminus D}) \\
&= \iota^* \csm(w'_*(\one_{W\smallsetminus D})) \\
&= \iota^* \csm(\one_{\Til V\smallsetminus E})\quad,
\end{align*}
where $\overset\dagger=$ holds by Theorem~6.2 in \cite{MR85k:14004}
(note that $j^*=j^!=\iota^!$ as both $D$ and $E$ are Cartier
divisors). We have $w'_*(\one_{W\smallsetminus D})
=\one_{V'\smallsetminus E}$ since $w'$ is an isomorphism off $D$.
Finally, using Proposition~\ref{birinv} and again the covariance of
$\csm$:
\begin{align*}
\csm(\cchi_{X,V}) &=\csm(e_* \cchi_{E,\Til V}) \\
& =e_*\csm(\cchi_{E,\Til V}) \\
& =e_* \iota^* \csm(\one_{\Til V\smallsetminus E})\\
&=\sigma_{X,V} (\csm(\one_{V\smallsetminus X}))
\end{align*}
as claimed.
\end{proof}

\begin{example}
Let $X\subseteq \Pbb^2$ be the line $L$ with embedded point $p$
considered in Example~\ref{nonred}. By Theorem~\ref{spethm},
\[
\sigma_{X,\Pbb^2}(\csm(\one_{\Pbb^2\smallsetminus X}))
=\csm(\cchi_{X,\Pbb^2})=\csm(\one_L+\one_p)=[L]+3[p]\quad.
\]
(This may of course be verified by hand using
Definition~\ref{specdefspec}.)  Here, $\csm(\one_{\Pbb^2\smallsetminus
X})=[\Pbb^2]+2[\Pbb^1]+[\Pbb^0]$.  Restricting this class to $L$ gives
$[L]+2[p]$; the specialization $\sigma_{X,\Pbb^2}$ picks up an extra
$[p]$ due to the embedded point.
\end{example}

\begin{remark}\label{regemb}
By Lemma~\ref{spediv}, if $i: X\to V$ embeds $X$ as a
Cartier divisor in $V$, then $\sigma_{X,V}$ acts as the pull-back to
$X$, so that in this case Theorem~\ref{spethm} states that
\[
\csm(\cchi_{X,V})=i^* \csm(\one_{V\smallsetminus X})\quad.
\]
This is not the case if $i: X\to V$ is a regular embedding of higher
codimension.  If the embedding is regular, then the exceptional
divisor $E\subset \Til V$ may be identified with the projective normal
bundle to $X$, and as such it has a universal quotient bundle
$\cN$. Tracing the argument proving Theorem~\ref{spethm} shows that
\[
i^* \csm(\one_{V\smallsetminus X})=e_*\left( c_{\text{top}}(\cN)
\cap \csm(\cchi_{E,\Til V})\right)\quad;
\]
and in fact
\[
i^* \csm(\one_{V\smallsetminus X})=d_*\left(
c_{\text{top}}\left(\frac{d^*N_XV} {N_DV}\right) \cap
\csm(\cchi_{D,W})\right)
\]
for any $W$, $D$, etc.~as in Definition~\ref{maindef}.
\end{remark}

\begin{warning}
When $X\subset V$ is a regular embedding, so that the blow-up morphism
$\Til V \to V$ is a local complete intersection morphism
(cf.~\S6.7 in \cite{MR85k:14004}) one may be
tempted to define a specialization of classes from $V$ to $X$ by
pulling back to the blow-up, restricting to the exceptional divisor,
and pushing forward to $X$. However, the pull-back of
$\csm(\one_{V\smallsetminus X})$ may not agree with $\csm(\one_{\Til
V\smallsetminus E})$: the blow-up of $V=\Pbb^2$ at $X=$ a point
already gives a counterexample. (Corollary~4.4 in \cite{MR2504753}
provides a condition under which a similar pull-back formula does
hold.) Thus, this operation does not agree in general with the class
defined in Definition~\ref{specdefspec}.
\end{warning}

\begin{remark}\label{distdeccla}
Since $\cchi_{X,V}$ admits a distinguished decomposition $\sum_m m
\cchi_{X,V}^{\epsilon_m}$ (see Remark~\ref{distdecfun}), so does the
specialization of $\csm(\one_{V\smallsetminus X})$:
\[
\sigma_{X,V}( \csm(\one_{V\smallsetminus X}))=\sum_m m\,
\csm(\cchi_{X,V}^{\epsilon_m})\quad.
\]
In the classical case of specializations along the zero set of a
function, this means that every eigenspace of monodromy carries a
well-defined piece of the Chern class, a fact we find intriguing, and
which we are not sure how to interpret for more general $X\subset
V$. In fact, any $\alpha: \Zbb \to \Zbb$ determines a class
$\csm(\cchi_{X,V}^\alpha)$; Theorem~\ref{spethm} gives an
interpretation of this class for $\alpha=$ identity.  It would be
interesting to interpret this class for more general $\alpha$.
\end{remark}

\subsection{Specialization to the central fiber of a morphism}\label{specenfi}
To interpret Theorem~\ref{spethm} in terms of standard
specializations, consider the template situation in which $V$ fibers
over a nonsingular curve $T$, and $X$ is the fiber over a marked point
$0\in T$:
\[
\xymatrix{
X \ar@{^(->}[r] \ar[d] & V \ar[d]^v \\
\{0\} \ar@{^(->}[r] & T 
}
\]
Further assume for simplicity that the restriction of $v$ to
$V\smallsetminus X$ is a trivial fibration:
\[
V\smallsetminus X\cong (T\smallsetminus \{0\})\times V_t\quad,
\]
with nonsingular (`general') fiber $V_t$. (In particular,
$V\smallsetminus X$ is nonsingular, as in \S\ref{basicca}.)  Here $X$
is a Cartier divisor in $V$, so the specialization of
Definition~\ref{specdefspec} agrees with ordinary pull-back
(Lemma~\ref{spediv}). We can use this pull-back to specialize
classes from the general fiber~$V_t$ to the special fiber
$v^{-1}(0)=X$, as follows.  The pull-back $i^*: A_*V \to A_*X$ factors
through $A_*(V\smallsetminus X)$ via the basic exact sequence of
Proposition~1.8 in \cite{MR85k:14004}: indeed, for $\alpha\in A_*X$,
$i^*i_*(\alpha)=c_1(N_XV)\cap \alpha=0$ as $N_XV$ is trivial.  We may
then define a specialization morphism by composing this morphism with
flat pull-back:
\[
\sigma_*: A_* V_t \to A_*(V\smallsetminus X) \to A_*X\quad.
\]

\begin{corol}\label{specialca}
In the specialization situation just described,
\[
\csm(\cchi_{X,V})=\sigma_* (c(TV_t)\cap [V_t])\quad.
\]
\end{corol}

\begin{proof}
Since the normal bundle of $V_t$ in $V\smallsetminus X$ is trivial,
and abusing notation slightly, $c(TV_t)\cap [V_t]=c(T(V\smallsetminus
X)\cap [V_t])$. Hence, by definition of $\sigma_*$ we have
\[
\sigma_* (c(TV_t)\cap [V_t])=i^*(\tau)\quad,
\]
where $\tau\in A_*V$ is any class whose restriction to
$A_*(V\smallsetminus X)$ equals $c(T(V\smallsetminus X)\cap
[V\smallsetminus X])$.  We claim that $\tau=\csm(\one_{V\smallsetminus
X})$ is such a class; it then follows that
\[
\sigma_* (c(TV_t)\cap [V_t])=i^*(\csm(\one_{V\smallsetminus X}))
=\sigma_{X,V}(\csm(\one_{V\smallsetminus X}))=\csm(\cchi_{X,V})
\quad,
\]
by Theorem~\ref{spethm} (and Lemma~\ref{spediv}).  Our claim is
essentially immediate if $V$ itself is nonsingular, but as we are only
assuming $V\smallsetminus X$ to be nonsingular we have to do a bit of
work. Again consider a resolution $w: W\to V$ as in
Definition~\ref{maindef}.  Let $w':(W\smallsetminus D) \to
(V\smallsetminus X)$ be the restriction (an isomorphism by
hypothesis). Also, let $i': V\smallsetminus X \to V$ and $j':
W\smallsetminus D \to W$ be the open inclusions. Thus we have the
fiber diagram
\[
\xymatrix{
W\smallsetminus D \ar@{^(->}[r]^-{j'} \ar[d]_{w'} & W \ar[d]^w \\
V\smallsetminus X \ar@{^(->}[r]^-{i'} & V
}
\]
with $i'$ and $j'$ flat, and $w$, $w'$ proper. We have
\begin{align*}
c(T(V\smallsetminus X))\cap [V\smallsetminus X] 
&= w'_* c(T(W\smallsetminus D))\cap [W\smallsetminus D] \\
&\overset{(1)}= w'_* {j'}^* c(TW(-\log D))\cap [W] \\
&\overset{(2)}= {i'}^* w_* c(TW(-\log D))\cap [W] \\
&\overset{(3)}= {i'}^* w_* \csm(\one_{W\smallsetminus D}) \\
&= {i'}^* \csm(\one_{V\smallsetminus X})
\end{align*}
as claimed. Equality (1) holds by definition of $TW(-\log D)
=(\Omega^1_W(\log D))^\vee$; equality~(2) follows from Proposition~1.7
in \cite{MR85k:14004}; equality (3) is again the computation of the
CSM class of the complement of a divisor with normal crossings
mentioned in the proof of Theorem~\ref{spethm}.
\end{proof}

Summarizing, in the strong specialization situation detailed above,
$\cchi_{X,V}$ is the constructible function on the central fiber $X$
corresponding via MacPherson's natural transformation to the
specialization of the Chern class of the general fiber.

\begin{remark}\label{spegen}
The hypothesis of triviality of the family away from $0\in T$ is not
necessary, if specialization is interpreted appropriately. 
Let $v: V\to T$ be a morphism to a nonsingular curve, and let
$X=v^{-1}(0)$ as above; assume our blanket hypothesis that
$V\smallsetminus X$ is nonsingular, but no more. 
As $V$ is nonsingular
away from $X$, we may assume $v$ is smooth, after replacing $T$ with a
neighborhood of $0$ in $T$.  The fiber $V_t$ is then smooth for all
$t\ne 0$ in $T$, and since $N_{V_t}V$ is trivial, we get
\[
\csm(\one_{V_t}) =c(V_t)\cap [V_t]=V_t\cdot (c(T(V\smallsetminus X))
\cap [V\smallsetminus X])\quad;
\]
since $\csm(\one_V)$ restricts to $c(T(V\smallsetminus X))
\cap [V\smallsetminus X]$ on $V\smallsetminus X$, this gives
\begin{equation*}
\tag{$\dagger$}
\csm(\one_{V_t}) =V_t\cdot \csm(\one_V)\quad.
\end{equation*}
This computation fails for $t=0$, as the fiber $X$ over $0$ is (possibly)
singular. However, the classes
$\csm(\one_V)$ and $\csm(\one_{V\smallsetminus X})$ have the same
specialization to $X$: indeed, their difference is supported on $X$,
and $N_XV$ is trivial in the case considered here. Thus, in this case 
Theorem~\ref{spethm} gives
\begin{equation*}
\tag{$\ddagger$}
\csm(\cchi_{X,V})=X\cdot \csm(\one_V)\quad.
\end{equation*}
Comparing ($\dagger$) and ($\ddagger$), we may still view $\cchi_{X,V}$ 
as the limit of the constant $\one_{V_t}$ as $t\to 0$.
\end{remark}

\begin{remark}
We can also consider specializations arising from maps $v: V\to T$,
allowing $T$ to be a nonsingular variety of arbitrary dimension with a
marked point $0$, and $X=v^{-1}(0)$ a local complete intersection of
codimension $\dim T$, with trivial normal bundle. We still assume
$V\smallsetminus X$ to be nonsingular.  Using the formula in
Remark~\ref{regemb}, the specialization $i^* \csm(\one_{V
\smallsetminus X})$ may be written as
\[
i^* \csm(\one_{V\smallsetminus X}) = d_* \left((-D)^{\dim T-1}\cdot
\csm(\cchi_{D,W})\right)\quad,
\]
where $W$, $D$, $d$ are as in Definition~\ref{maindef}.
Indeed, since $N_XV$ is trivial, then
\[
c_{\text{top}}\left(\frac{d^* N_XV}{N_DV}\right)=
\text{term of codimension $(\dim T-1)$ in }
\frac 1{1+D}=(-D)^{\dim T-1}\quad.
\] 

Note that in this situation $d_*(-D)^{\dim T-1}\cdot [D]=[X]$: indeed,
the Segre class of $D$ in $W$ pushes forward to the Segre class of $X$
in $V$, by the birational invariance of Segre classes.

As mentioned in the introduction, Sch\"urmann has considered
specialization to a complete intersection, by iterating applications
of Verdier specialization (\cite{Schu}, Definition~3.6). It would be
interesting to establish a precise relation between Sch\"urmann's
specialization and the formula given above.
\end{remark}


\section{Example: pencils of curves}\label{pencils}

\subsection{}\label{spesu}
Pencils of hypersurfaces give rise to specializations, as follows.

Consider a pencil of hypersurfaces in a linear system in a nonsingular
variety $V'$. Let the pencil be defined by the equation
\[
F+tG=0
\]
where $F$ and $G$ are elements of the system, and $t\in k$. Assume
that for any $t\neq 0$ in a neighborhood of $0$, $V_t=\{F+tG=0\}$ is
nonsingular.  We can interpret this datum as a specialization by
letting $V$ be the correspondence
\[
V=\{(p,t)\in V'\times \Abb^1 \,|\, F(p)+tG(p)=0\}\quad.
\]
This is endowed with a projection $v: V \to T=\Abb^1$; after removing
a finite set of $t\ne 0$ from~$\Abb^1$ (that is, those $t\ne 0$ for
which the fiber $F+tG=0$ is singular), we reach the standard situation
considered in \S\ref{specenfi} (Remark~\ref{spegen}):
\begin{itemize}
\item $T$ is a nonsingular curve, and $0\in T$;
\item $v: V\to T$ is a surjective morphism, and $X=v^{-1}(0)$;
\item $V\smallsetminus X$ is nonsingular, and $v|_{V\smallsetminus X}$
is smooth.
\end{itemize}
Note that $V$ may be singular along $X$; in fact, the singularities of $V$
are contained in the base locus of the pencil and in the singular locus of $X$.

As proven in \S\ref{mainpro}, 
\[
\csm(\cchi_{X,V})=X\cdot \csm(\one_V)\quad,
\]
for $\cchi_{X,V}$ as in Definition~\ref{maindef}, and this class can
be viewed as the limit as $t\to 0$ of the classes $c(TV_t)\cap [V_t]$
of the general fibers.  For example, the degree $\int
\csm(\cchi_{X,V})$ equals $\int c(TV_t) \cap [V_t]=\chi(V_t)$, the
Euler characteristic of the general fiber.

The following are immediate consequences of the definition.
\begin{itemize}
\item Let $p\in X$ be a point at which $X$ (and hence $V$) is
nonsingular.  Then $\cchi_{X,V}(p)=1$.
\item More generally, let $p$ be a point such that there exists a
neighborhood $U$ of $p$ in $V$ such that $U$ is nonsingular, and
$U\cap X$ is a normal crossing divisor in $U\cap V$.  Then
$\cchi_{X,V}(p)=0$ if $p$ is in the intersection of two or more
components of $X$, and $\cchi_{X,V}(p)=m$ if $p$ is in a single
component of $X$, of multiplicity $m$.
\end{itemize}

Since $\cchi_{X,V}(p)=1$ at nonsingular points of $X$,
$\one_X-\cchi_{X,V}$ is supported on the singular locus of $X$. This
function has a compelling interpretation, see (*) in the introduction
and~\S\ref{wmather}. The following example illustrates a typical
situation.

\begin{example}\label{milnorno}
Let $X$ be a reduced hypersurface with isolated singularities in a
nonsingular variety $V'$, and assume a general element $V_t$ of the
linear system of $X$ is nonsingular, and avoids the singularities of
$X$. Then $(-1)^{\dim V} (1-\cchi_{X,V}(p))$ equals the Milnor number
$\mu_X(p)$ of $X$ at $p$.

Indeed, as the matter is local, after a resolution we may assume that
$X$ is complete and $p$ is its only singularity. Consider the pencil
through $X$ and a general element $V_t$ of its linear system.  Then by
linearity of $\csm$, and since $\one_X-\cchi_{X,V}$ is $0$ away from
$p$,
\begin{multline*}
(-1)^{\dim V} (1-\cchi_{X,V}(p))
=(-1)^{\dim V} \int \csm(\one_X(p)-\cchi_{X,V}(p))\\
=(-1)^{\dim V} \int \csm(\one_X - \cchi_{X,V})
=(-1)^{\dim V}(\int \csm(\one_X)-\int \csm(\cchi_{X,V}))\quad.
\end{multline*}
Now $\int \csm(\one_X)=\chi(X)$ (as recalled in \S\ref{basicca}), and 
$\int \csm(\cchi_{X,V})=\int \csm(V_t)=\chi(V_t)$ as observed above.
Thus,
\[
(-1)^{\dim V} (1-\cchi_{X,V}(p))
=(-1)^{\dim V} (\chi(X)-\chi(V_t))\quad,
\]
and this is well-known to equal the Milnor number of $X$ at $p$
(see e.g., \cite{MR949831}, Corollary~1.7).\smallskip

The above formula is a particular case of the formula 
$\mu=(-1)^{\dim V}(\one_X-\cchi_{X,V})$ mentioned in the introduction,
and discussed further in \S\ref{wmather}.
\end{example}

Near points of $X$ away from the base locus, $V$ is trivially
isomorphic to $V'$, hence it is itself nonsingular; the specialization
function is then computed directly by an embedded resolution of
$X$. Along the base locus, $V$ itself may be singular, and it will
usually be necessary to resolve $V$ first in order to apply
Definition~\ref{maindef} (or Proposition~\ref{birinv}).

In the following subsections we illustrate this process in a few
simple examples, for pencils of curves.

\subsection{Singular points on curves}\label{sincurves}
Let $X$ be a curve with an ordinary multiple point $p$ of multiplicity
$m$, and assume that this point is not in the base locus of the
pencil. As pointed out above, $\cchi_{X,V}(p)$ may be computed by
considering the embedded resolution of $X$ over $p$.  This may be
schematically represented as:
\begin{center}
\includegraphics[scale=.5]{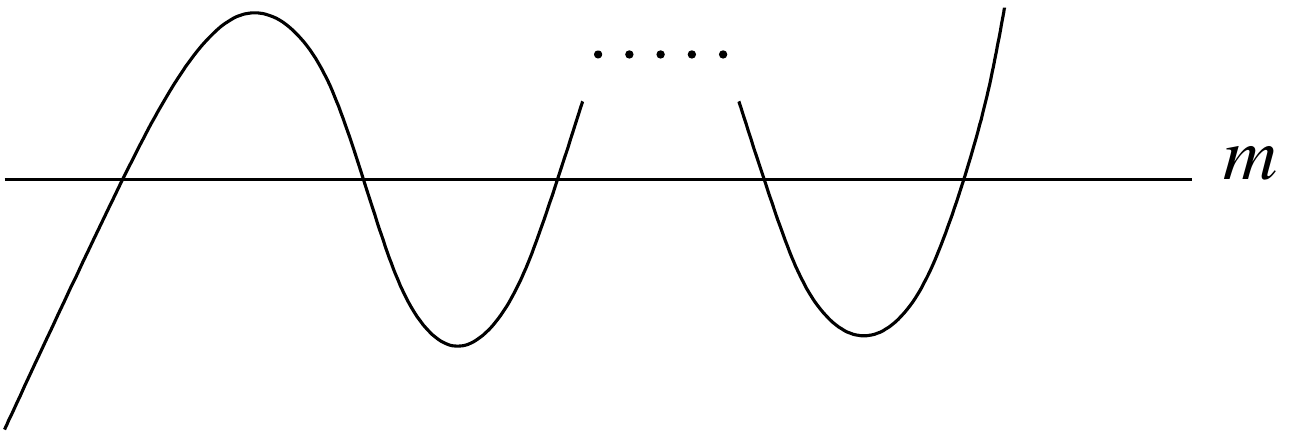}
\end{center}
We have one exceptional divisor, a $\Pbb^1$, meeting the proper
transform of the curve $m$ times and appearing with multiplicity
$m$. As the Euler characteristic of the complement of the $m$ points
of intersection in the exceptional divisor is $2-m$, we have that
$\cchi_{X,V}(p)=m(2-m)$.

More generally, let $X$ be a plane curve with an isolated singular
point $p$, again assumed to be away from the base locus of the pencil.
If the embedded resolution of $X$ has exceptional divisors $D_i$ over
$p$, $D_i$ appears with multiplicity $m_i$, and meets the rest of the
full transform of $X$ at $r_i$ points, then
\[
\cchi_{X,V}(p)=\sum_i m_i(2-r_i)\quad:
\]
indeed, each $D_i$ is a copy of $\Pbb^1$, and the Euler characteristic
of the complement of $r_i$ points in $\Pbb^1$ is $2-r_i$.  It follows
(cf.~Proposition~\ref{milnorno}) that the Milnor number of $p$ is
$1-\sum_i m_i(2-r_i)$, yielding a quick proof of this well-known
formula (see \cite{MR88a:14001}, \S8.5, Lemma~3, for a discussion of
the geometry underlying this formula over $\Cbb$).

For example, the resolution graph of an ordinary cusp is:
\begin{center}
\includegraphics[scale=.5]{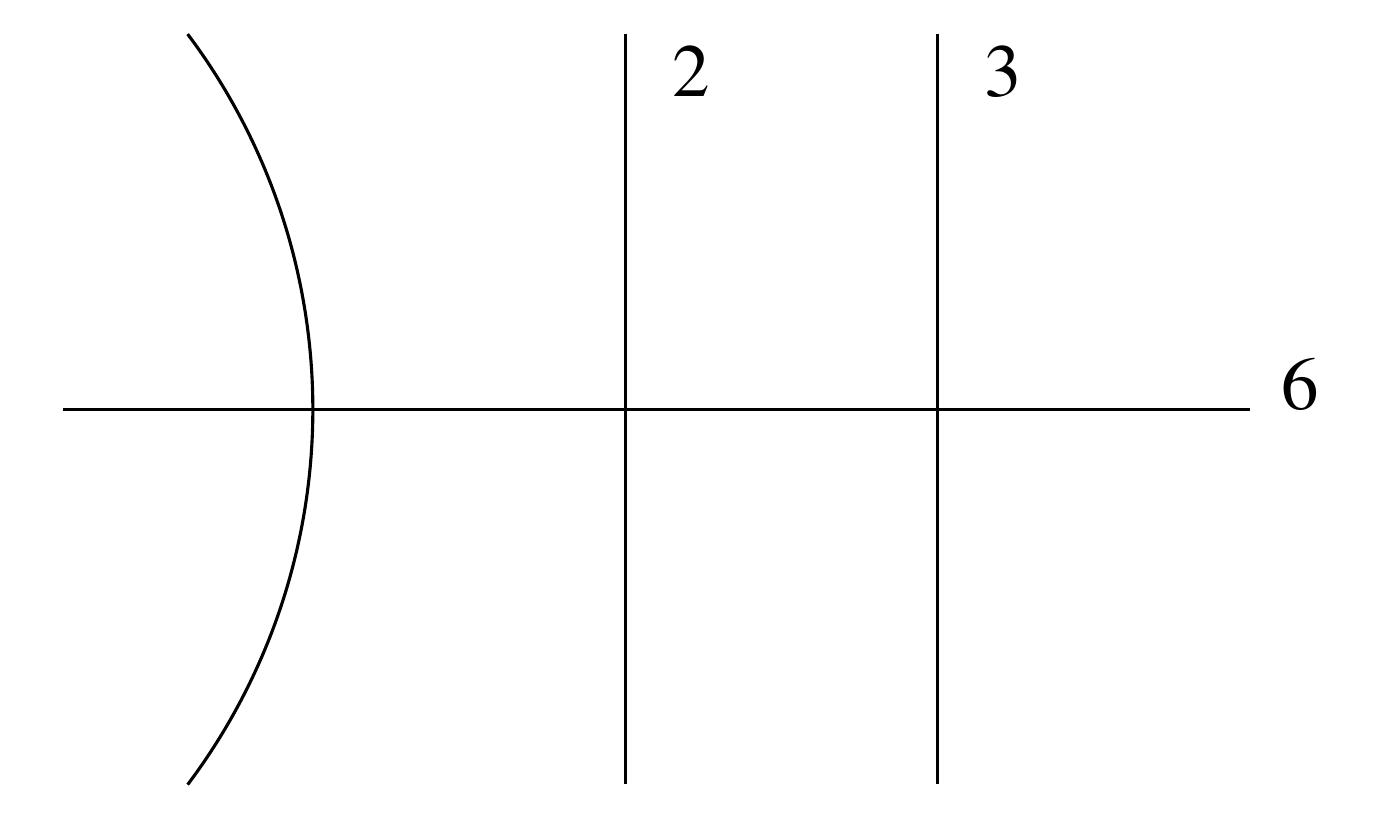}
\end{center}
where the numbers indicate the multiplicity of the exceptional
divisors.  It follows that
$\cchi_{X,V}(p)=2\cdot(2-1)+3\cdot(2-1)+6\cdot(2-3)=-1$ at an ordinary
cusp.

\subsection{Cuspidal curve, cusp in the base locus}\label{cuspta}
Typically, $V$ is {\em singular\/} at base points of the system at
which $X$ is singular. Subtleties in the computation of the
specialization function arise precisely because of this phenomenon.

Again consider a pencil centered at an ordinary cusp $p$, but such
that the general element of the pencil is nonsingular at $p$ and
tangent to $X$ at $p$:
\begin{center}
\includegraphics[scale=.5]{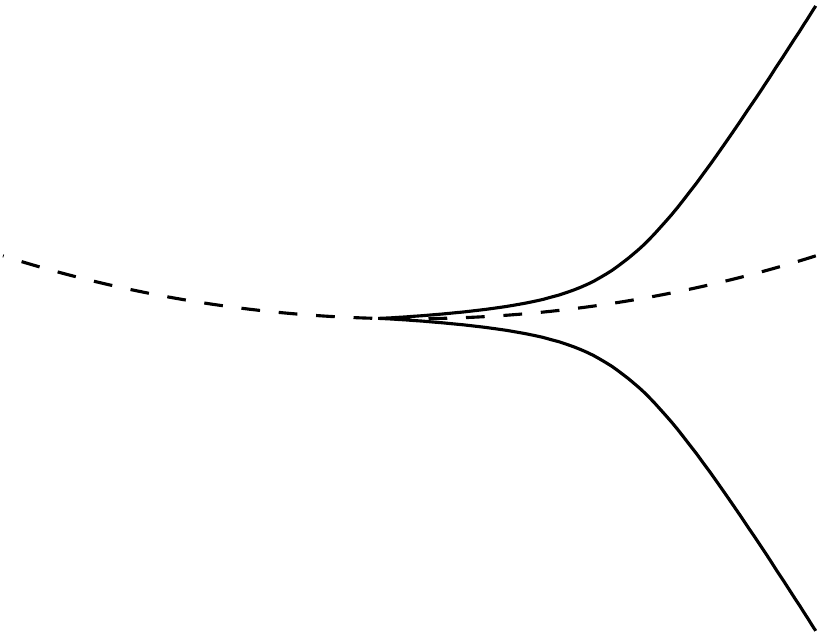}
\end{center}
A local description for the correspondence $V$ near $p$ is
\[
(y^2-x^3)-t y=0\quad.
\]
This may be viewed as a singular hypersurface in $\Abb^3_{(x,y,t)}$
and is resolved by a single blow-up at $(0,0,0)$ (as the reader may
check).  One more blow-up produces a divisor with normal crossings as
needed in Definition~\ref{maindef}, with the same resolution graph as
in \S\ref{sincurves}:
\begin{center}
\includegraphics[scale=.5]{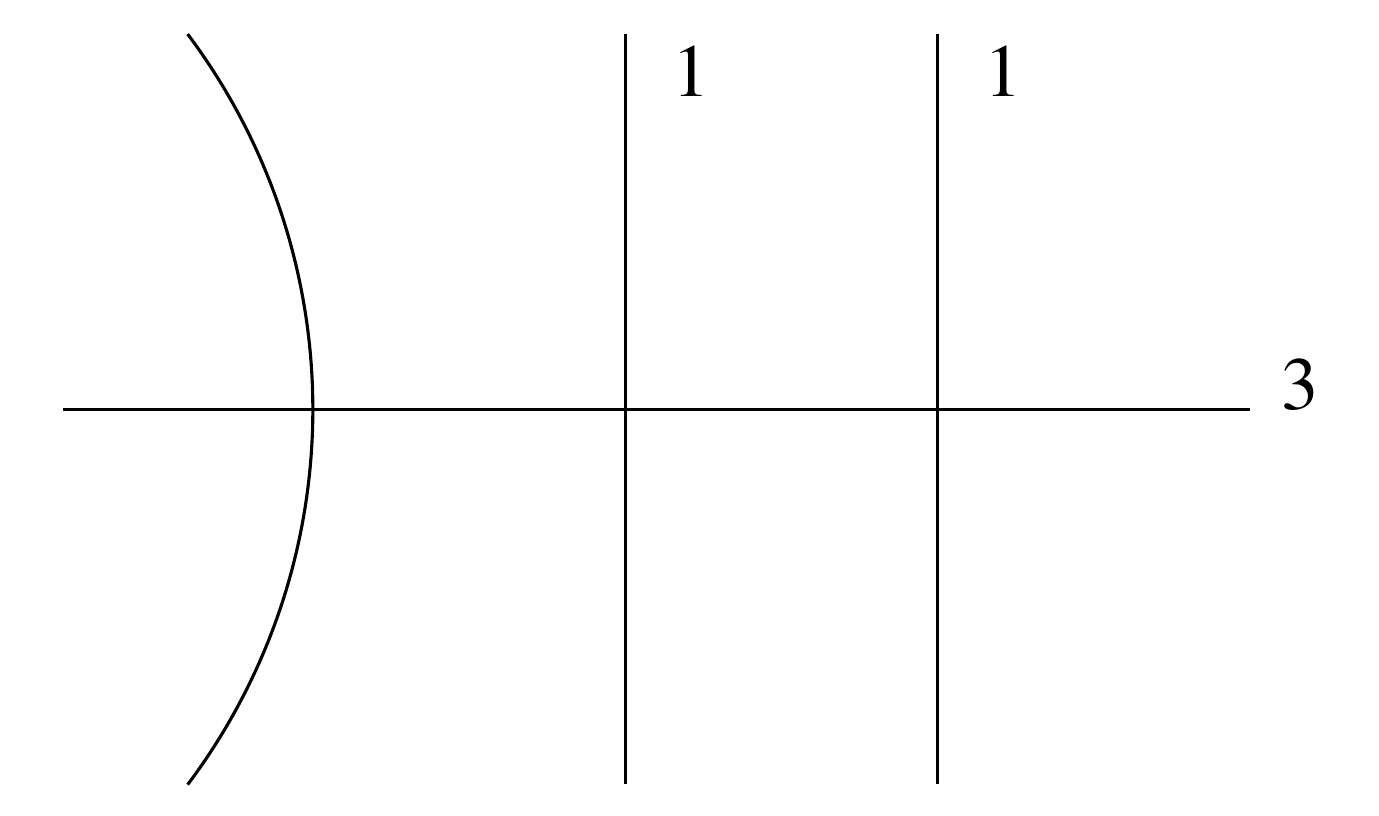}
\end{center}
but with different multiplicities, as indicated. This does not affect the 
value of $\cchi_{X,V}(p)$:
\[
\cchi_{X,V}(p)=1\cdot (2-1)+1\cdot (2-1)+3\cdot (2-3)=-1\quad.
\]

In fact, it is clear `for specialization reasons' that the value of
$\cchi_{X,V}(p)$ at an isolated singularity $p$ of $X$ is the same
whether $p$ is in the base locus of the pencil or not, since (as
pointed out at the beginning of this section) the Euler characteristic
of $X$ weighted according to $\cchi_{X,V}$ must equal the Euler
characteristic of the general element of the pencil, and this latter
is unaffected by the intersection of $X$ with the general element.  It
is however interesting that the geometry of the resolution {\em is\/}
affected by the base locus of the pencil: the total space of the
specialization is smooth near $p$ if the cusp is not in the base locus
(as in \S\ref{sincurves}), while it is singular if the cusp is in the
base locus (as in this subsection).  Any difference in the normal
crossing resolution due to these features must compensate and produce
the same value for~$\cchi_{X,V}(p)$.  The next example will illustrate
that this is not necessarily the case for non-isolated singularities.

Note also that the `distinguished decomposition' of $\cchi_{X,V}$ (or
of its motivic counterparts) do tell these two situations apart: for
example, with notation as in Remark~\ref{distdecfun},
$\cchi_{X,V}^{\epsilon_3}=\one_p$ for the cuspidal curve in
\S\ref{sincurves} (cusp $p$ away from the base locus), while
$\cchi_{X,V}^{\epsilon_3}=-\one_p$ for the cuspidal curve in this
subsection (cusp on the base locus).

\subsection{Non-isolated singularities}
Let $V'$ be a nonsingular surface, and $X\subset V'$ be a (possibly
multiple, reducible) curve. Consider the pencil between $X$ and a
nonsingular curve~$Y$ meeting a component of multiplicity $m\ge 1$ in
$X$ transversally at a general point~$p$. View this as a
specialization, as explained in~\S\ref{spesu}.  Then $\cchi_{X,V}(p)=1$,
regardless of the multiplicity~$m$.

To verify this, we may choose analytic coordinates $(x,y)$ so that $p$
is the origin, $X$ is given by $x^{n+1}$ for $n=m-1$, and $Y$ is
$y=0$; the correspondence $V$ is given by
\[
x^{n+1}-yt=0
\]
in coordinates $(x,y,t)$.

If $n=0$, then both $X$ and $V$ are nonsingular, and $\cchi_X(p)=1$ as
seen above.  If $n>0$, $V$ has an $A_n$ singularity at the origin, and
its resolution is classically well-known: the exceptional divisors are
$\Pbb^1$s, linked according to the $A_n$ diagram,
\[
\circ - \circ - \cdots - \circ \quad.
\]
The only work needed here is to keep track of the multiplicities of
the components in the inverse image of $t=0$. The reader may check
that the pull-back of $t=0$ in the resolution consists of the proper
transform of the original central fiber, with multiplicity $n+1$, and
of a chain of smooth rational curves, with decreasing multiplicities:
\begin{center}
\includegraphics[scale=.5]{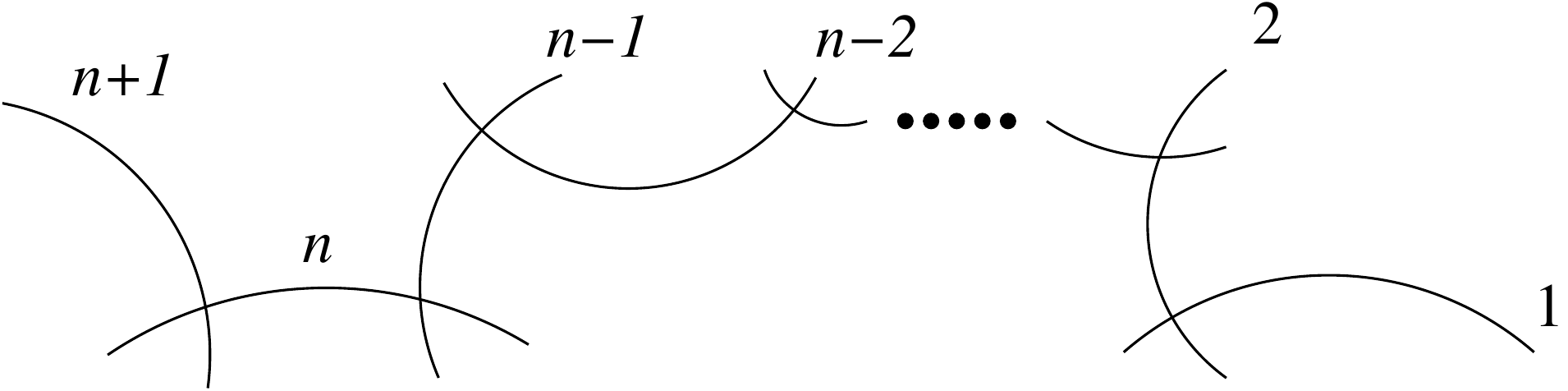}
\end{center}
The contribution to $\cchi_{X,V}(p)$ of all but the right-most
component in this diagram is $0$, because the Euler characteristic of
the complement of $2$ points in $\Pbb^1$ is $0$. The right-most
component contributes~$1$, since it appears with multiplicity $1$ and
it contributes by the Euler characteristic of the complement of one
point in $\Pbb^1$, that is $1$. Thus $\cchi_{X,V}(p)=1$, as claimed.

Note that $\cchi_{X,V}=m$ at a general point of a component of
multiplicity $m$ in $X$.  Thus, the effect of $p$ being in the base
locus is to erase the multiplicity information, provided that the
general element $Y$ of the pencil meets $X$ transversally at $p$.

This is in fact an instance of a general result, proven over $\Cbb$ in
all dimensions and for arbitrarily singular $X$ by A.~Parusi\'nski and
P.~Pragacz in Proposition~5.1 of~\cite{MR2002g:14005}. The proof given
in this reference uses rather delicate geometric arguments (for
example, it relies on the fact that a Whitney stratification satisfies
Thom's $a_f$ condition, \cite{MR1242881}).  It would be worthwhile
constructing a direct argument from the definition for $\cchi_{X,V}$
given in this paper, and valid over any algebraically closed field of
characteristic~$0$.


\section{Weighted Chern-Mather classes, and a resolution formula}\label{wmather}

\subsection{}
Let $V$ be a nonsingular variety, and let $Y\subset V$ be a closed
subscheme.  In \cite{MR1819626} we have considered the {\em weighted
Chern-Mather class\/} of $Y$,
\[
\cwm(Y):=\sum_i (-1)^{\dim Y-\dim Y_i} m_i j_{i*} \cma(Y_i)\quad,
\]
where $Y_i$ are the supports of the components of the normal cone of 
$Y$ in $V$, $m_i$ is the multiplicity of the component over $Y_i$, and 
$\cma$ denotes the ordinary Chern-Mather class. 
Also, $\dim Y$ is the largest dimension of a component of $Y$. 
Up to a sign, $\cwm$ is the same as the {\em Aluffi class\/} 
of~\cite{MR2600874}.
With the definition given above, if 
$Y$ is irreducible and reduced,
then $\cwm(Y)$ equals $\cma(Y)$; in particular, $\cwm(Y)=c(TY)\cap
[Y]$ if $Y$ is nonsingular. (However, with this 
choice of sign the contribution of a component $Y_i$ depends on the
dimension of the largest component of~$Y$.)

Consider the case in which $Y$ is the {\em singularity subscheme\/}
of a hypersurface $X$ in $V$: if $f$ is a local generator for the ideal of $X$, then 
the ideal of $Y$ in $V$ is locally generated by $f$ and the partials of $f$.

\begin{prop}\label{wmaprop}
Let $X$ be a hypersurface in a nonsingular variety $V$, and let $Y$ be its
singularity subscheme. Then
\begin{equation*}
\tag{*}
(-1)^{\dim Y} \cwm(Y)=(-1)^{\dim V} \csm(\one_{X} - \cchi_{X,V})
\end{equation*}
in $A_*X$.
\end{prop}

Since $\csm(\cchi_{X,V})$ admits a natural multiplicity decomposition
(Remark~\ref{distdeccla}), so do the class $\cwm(Y)$ and its degree, a
Donaldson-Thomas type invariant (\cite{MR2600874}, Proposition~4.16). 
I.e., monodromy induces a decomposition of these invariants.\smallskip

\begin{remark}
The relation in Proposition~\ref{wmaprop} is a $\csm$-counterpart of the 
identity $\mu=(-1)^{\dim V}
(\one_X-\cchi_{X,V})$ mentioned in the introduction, and is equivalent
to Theorem~1.5 in \cite{MR1819626}. The identity amounts to the
relation of~$\mu$ with the Euler characteristic of the perverse sheaf
of vanishing cycles. It goes at least as far back as \cite{MR0371889},
Theorem~4; it also follows from Proposition~5.1
in~\cite{MR2002g:14005}, and is equation (4)
in~\cite{MR2600874}. However, these references work over $\Cbb$, and
it seems appropriate to indicate a proof of (*) which is closer in
spirit to the content of this paper.
\end{remark}

\begin{proof}
By Theorem~\ref{spethm},
\[
\csm(\cchi_{X,V})=X\cdot ( \csm(\one_{V\smallsetminus X}))
\]
(cf.~Lemma~\ref{spediv}). Thus
\begin{align*}
\csm(\one_X - \cchi_{X,V})
& =\csm(\one_X) - X\cdot \csm(\one_V) + X\cdot \csm(\one_X)\\
& =(1+X)\cdot \csm(\one_X)-c(TV|_X)\cap [X] \\
& =c(\cL) \cap \left(\csm(\one_X)-\frac{c(TV|_X)}{1+X}\cap [X]\right)
\end{align*}
with $\cL=\cO(X)$. The class $\frac{c(TV|_X)}{1+X}\cap [X]$ is the
class of the virtual tangent bundle to $X$, denoted $\cfu(X)$ in
\cite{MR1819626}.  Thus,
\[
\csm(\one_X - \cchi_{X,V})=c(\cL) \cap \left(\csm(X)-\cfu(X)\right)\quad.
\]
By Theorem~1.2 in \cite{MR1819626}, it follows that
\[
\csm(\one_X-\cchi_{X,V})=(-1)^{\dim V-\dim Y} \cwm(Y)\quad,
\]
which is the statement.
\end{proof}

Both sides of the formula in Proposition~\ref{wmaprop}
make sense in a more general situation than the case in which $Y$ is the 
singularity subscheme of a hypersurface $X$: the class $\cwm(Y)$ is defined
for arbitrary subschemes of a nonsingular variety, and Definition~\ref{maindef}
also gives a meaning to $\cchi_{X,V}$ for any subscheme $X$ of a nonsingular
variety. 
Otherwise put, both sides of the identity $\mu = (-1)^{\dim V} (\one_X -
\cchi_{X,V})$ admit compelling generalizations: Definition~\ref{maindef}
does not require $X$ to be a hypersurface, and $\cwm(Y)$ is the class 
corresponding to a constructible function $\nu_Y$ defined for all $Y$
as a specific combination of local Euler obstructions (see Proposition~1.5 in
\cite{MR1819626}). This function has garnered some interest in 
Donaldson-Thomas theory, and is currently commonly known as 
{\em Behrend's function.\/} It is tempting to guess that a statement 
closely related to Proposition~\ref{wmaprop} may still hold for 
{\em any\/} $Y$ and a suitable choice of $X\subset V$ with singularities
along $Y$. (Of course $\one_X$ should be
replaced by $(\codim_XV) \one_X$ for the right-hand-side to be
supported on the singularities of $X$, see Example~\ref{sci}.)

\subsection{}
We end with an expression for the function~$\mu$ for the singularity
subscheme $Y$ of a hypersurface $X$ in terms of a resolution of
$X$. For this we need to work with $\Qbb$-valued constructible
functions; we do not know if a similar statement can be given over
$\Zbb$.

Assume $V$ is nonsingular, $X\subset V$ is a hypersurface, and
$Y\subseteq X$ is the singularity subscheme. Consider a morphism $w: W
\to V$ as in Definition~\ref{maindef}. Thus, $w^{-1}(X)$ is a divisor
$D$ with normal crossings and nonsingular components $D_\ell$,
$\ell\in L$; $m_\ell$ denotes the multiplicity of $D_\ell$ in $D$, and
$d: D\to X$ is the restriction of $w$.

The relative canonical divisor of $w$ is a combination $\sum_{\ell}
\mu_\ell D_\ell$ of components of $D$.

For $K\subseteq L$, we let
\[
D_K^\circ=(\cap_{k\in K} D_k)\smallsetminus (\cup_{\ell\not\in K}
D_\ell)\quad.
\]

\begin{prop}\label{Behrendexp}
With notation as above, $\mu$ is given by
\[
(-1)^{\dim X} d_*\left(\sum_{\ell\in L}\left(m_\ell-\frac
1{1+\mu_\ell}\right) \one_{D_\ell^\circ} -\sum_{K\subseteq L, |K|\ge
2} \frac 1{\prod_{k\in K} (1+\mu_k)} \one_{D_K^\circ}\right)\quad.
\]
\end{prop}

\begin{proof}
As discussed above, $\mu=(-1)^{\dim
X} (\cchi_{X,V}-\one_X)$. The given formula follows immediately from
\[
\cchi_{X,V}=d_* \sum_{\ell\in L} m_\ell \one_{D_\ell^\circ}\quad,
\]
which is a restatement of Definition~\ref{maindef}, and
\[
\one_X= d_* \left(\sum_{K\subseteq L, |K|\ge 1} \frac 1{\prod_{k\in K}
(1+\mu_k)} \one_{D_K^\circ}\right)\quad,
\]
which follows from a small generalization of \S4.4.3 in
\cite{MR1905328} (cf.~Theorem~2.1 and the proof of Theorem~3.1 in
\cite{MR2098642}).
\end{proof}

Again, it is tempting to guess that a similar expression may exist for 
Behrend's function of an arbitrary subscheme $Y$ of a nonsingular variety~$V$, 
for a suitable choice of a  corresponding pair $X\subset V$.



\end{document}